\documentclass{article}
\usepackage{amssymb}
\usepackage{amsfonts}
\usepackage{amsmath}
\usepackage{graphicx}
\usepackage[dvips]{color}
\usepackage{amsthm}
\usepackage{mathrsfs}
\usepackage{array}
\usepackage[latin1]{inputenc}

\renewcommand{\r}{\mathbb R}

\newcommand{\liz}{\mathcal L}

\newcommand{\qpn}{\mathbb Q_{p}^n}

\newcommand{\qp}{\mathbb Q_{p}}

\newcommand{\dd}{\mathcal D(\mathbb Q_p^n)}

\renewcommand{\d}{\partial }

\newcommand{\nep}{n_{\varepsilon}}

\newcommand{\Luno}{L^1(\mathbb Q_p^n)}

\renewcommand{\t}{\noindent}
\newtheorem{theorem}{Theorem}[section]
\theoremstyle{plain}

\newtheorem{corollary}[theorem]{Corollary}

\newtheorem{example}[theorem]{Example}

\newtheorem{lemma}[theorem]{Lemma}

\newtheorem{proposition}[theorem]{Proposition}
\newtheorem{remark}[theorem]{Remark}

\numberwithin{equation}{section}

\begin{document}
\title{The $p$-Adic Scattering equation}
\author{Jeanneth Galeano-Peñaloza, Oscar F. Casas-Sánchez}
\maketitle
\begin{abstract}
There are several techniques in classical case for some PDEs,  involving  the concept of entropy to show convergence of solutions to a steady state.  In this work we deal with the $p$-adic scattering equation and we try to adapt these methods to prove convergence of solutions.
\end{abstract}

\section{Preliminars}
There is a general strategy to prove convergence of solutions to PDEs towards a steady state, as follows, see \cite{Zamponi}. 

Suppose that we have an evolution equation with the form
$$\partial_t u + A(u(t))=0, \qquad t>0, \qquad u(0)=u_0,$$
where $u:(0, \infty)\to B$,  $B$ is some Banach space and $A:B\to B^*$ is some (nonlinear) mapping.  In addition, we have some functional $H=H[u]$, which we call an {\bf entropy functional} and a steady state $u_{\infty}$, i.e. a solution of $A(u)=0.$  The purpose here is to compute the {\it entropy production}, i.e. (minus) the time derivative of $H[u]$ and study its behavior along the solutions of the evolution equation.  If we can find a relation between the entropy production and the entropy itself, then Gronwall's lemma let us conclude that $H[u(t)]\to 0$ with some exponential rate as $t\to \infty$.

This technique is very common in the real case, and it works for some equations, among others, the Lotka-Volterra systems, Fokker-Planck equation, the scattering equation, and other some parabolic equations, which have several applications in biology, for example in systems related to population biology, ecological interactions, prey-predator systems, etc, see for example \cite{Biology}, \cite{Transport}.

\begin{theorem}[Th. VII.3 (Cauchy, Lipschitz, Picard)]\label{teo-Brezis}
	Let $E$ be a Banach space and $F:E\to E$ be a map such that
	$$||Fu-Fv||\le L||u-v|| \qquad \forall u, v\in E \ (L\ge 0).$$
	Then for all $u_0\in E$ there exists a unique $u\in C^1([0, \infty); E)$ such that 
	\begin{equation}
	\begin{cases}
	\frac{du}{dt}=Fu\quad t\in [0,\infty)\\
	u(0)=u_0.
	\end{cases}
	\end{equation}
\end{theorem}

\t The proof of this result can be found in \cite{Brezis}, and it is based on the Banach fixed point theorem.

Lets consider $E=L^1(\qpn)$ and $F:L^1(\qpn) \to L^1(\qpn)$, defined by 
\begin{equation}\label{operatorF}
F(n)= \int_{\qpn}K(y,x-y)n(t,y)d^ny  - k(x) n(t,x),
\end{equation}
where $0\le K(y,z)\in L^1 \cap L^{\infty}(\qp^{2n})$ and $0\le k(y):=\int_{\qpn} K(y,z)d^nz \in L^{\infty}(\qpn).$ We do not make  special assumption on the symmetry  of the cross-section $K(y,z)$ motivated by turning kernels that appear in some applications as bacterial movement.  We also suppose that $k$ and $K$ are independent of the time.

The operator $F$ is Lipschitz continuous, in fact:
\begin{align*}
||k(\cdot) n(\cdot)||_{L^1(\qpn)} &=\int_{\qpn}|k(x)n(t,x)|d^n x \\
&\le \int_{\qpn} \sup_{x\in \qpn}|k(x)| |n(t,x)|d^n x =||k||_{\infty} ||n||_{L^1(\qpn)},
\end{align*}
\begin{align*}
||\int_{\qpn}K(y, x-y) n(t,y)d^n y||_{\Luno} &= \int_{\qpn}\left|  \int_{\qpn}K(y, x-y) n(t,y)d^n y \right| d^n x\\
&\le \int_{\qpn}\int_{\qpn}K(y, x-y) |n(t,y)|d^n x d^n y \\
&=\int_{\qpn} |n(t,y)|\int_{\qpn} K(y, x-y)d^n x d^n y\\
&=\int_{\qpn}|n(t,y)| k(x-y) d^n y \\
&\le \int_{\qpn} |n(t,y)| \sup_{x\in \qpn} |k(x)| d^ny
\le ||k||_{\infty} ||n||_{\Luno},
\end{align*}
then $||F(n)||_{\Luno} \le 2||k||_{\infty} ||n||_{\Luno}\le C ||n||_{\Luno}.$

\section{Existence of solutions of the Scattering equation}	
The following results are the $p$-adic analogs to the classical  given in \cite{Biology}. The scattering equation has the form
\begin{equation}\label{scattering}
\frac{\d}{\d t}n(t,x) + k(x)n(t, x) = \int_{\qpn} K(y, x-y)n(t,y) d^n y, \qquad t\ge 0, x\in \qpn,
\end{equation}
the initial data $n^0\in L^1(\qpn)$, and we assume that
\begin{equation}\label{relacionK-k}
K(y,z)\ge 0, \qquad \qquad k(y):=\int_{\qpn} K(y,z)d^nz \in L^{\infty}(\qpn).
\end{equation}

\begin{lemma}
	The problem \eqref{scattering}-\eqref{relacionK-k} has a unique solution $n\in C^1([0, \infty), \Luno)$, and satisfies the following properties:
	\begin{equation}\label{eq--9.3}
	n^0\ge 0 \Rightarrow n\ge 0,
	\end{equation}
	\begin{equation}\label{eq--9.4}
\int_{\qpn} n(t,x)d^nx = \int_{\qpn}n^0(x)d^nx \quad \forall t\ge 0,
\end{equation}
	\begin{equation}\label{eq--9.5}
\int_{\qpn} |n(t,x)|d^nx \le  \int_{\qpn}|n^0(x)|d^nx \quad \forall t\ge 0.
\end{equation}
\end{lemma}

\begin{proof}
According to Theorem \ref{teo-Brezis} and since the operator defined in \eqref{operatorF} is Lipschitz continuous, the equation \eqref{scattering} has a unique solution $n\in C^1([0, \infty), \Luno)$.  In order to prove the properties, we write 
\begin{align*}
\frac{\d}{\d t}n(t,x) &=  \int_{\qpn} K(y, x-y)n(t,y) d^n y - k(x)n(t, x)\\
\int_{\qpn} \frac{\d}{\d t}n(t,x) d^n x&= \int_{\qpn} \int_{\qpn} K(y, x-y)n(t,y) d^n y - k(x)n(t, x) d^nx.
\end{align*}
By using the Lebesgue dominated convergence theorem and Fubini's theorem
\begin{align*}
\frac{d}{dt} \int_{\qpn} n(t,x) d^n x &= \int_{\qpn} \int_{\qpn} K(y, x-y)n(t,y) d^ny d^nx - \int_{\qpn} k(x)n(t, x) d^nx\\
&=\int_{\qpn}n(t,y) \int_{\qpn} K(y, x-y) d^nx d^ny \\
&\hspace{2cm}- \int_{\qpn} n(t, x) \int_{\qpn}K(x,z)d^nz d^nx\\
&=\int_{\qpn}n(t,y) \int_{\qpn} K(y, x-y) d^nx d^ny \\
&\hspace{2cm}- \int_{\qpn} n(t, x) \int_{\qpn}K(x,z-x)d^nz d^nx=0.
\end{align*}

\t It means that $\int_{\qpn} n(t,x) d^n x $ does not depend on $t$, therefore we obtain \eqref{eq--9.4}
\begin{equation*}
\int_{\qpn} n(t,x) d^n x =\int_{\qpn} n^0(x) d^n x , \quad \forall t\ge 0,
\end{equation*}
which  is known like {\it mass conservation law}.

In order to prove \eqref{eq--9.5}, consider $H_{\delta}(\cdot)$ a family of smooth, no-decreasing and convex functions such that $H_{\delta}'(\cdot)\le 1$ and $H_{\delta}(\cdot) \nearrow H(\cdot) = sgn_+(\cdot)$.  Such function $H$ is known as {\it entropy}, and we have to calculate the {\it entropy production} $P[n]$ defined as $P[n]=-\frac{d}{dt}H[n]$, for definitions see \cite{Zamponi}.

\begin{align*}
\frac{d}{dt} H_{\delta}(n(t,x))&= H'_{\delta}(n(t,x))\frac{\partial}{\partial t} n(t,x)\\
&= H'_{\delta}(n(t,x)) \left(  \int_{\qpn}K(y,x-y)n(t,y)dy  - k(x) n(t,x) \right)\\
\frac{d}{dt} H_{\delta}(n(t,x))  &+ H'_{\delta}(n(t,x))k(x) n(t,x)  =H'_{\delta}(n(t,x))\int_{\qpn}K(y,x-y)n(t,y)dy \\
&\le \int_{\qpn}K(y,x-y)n_+(t,y)dy 
\end{align*}
Taking limit as $\delta \to 0$ 
\begin{align*}
\frac{\partial}{\partial t} n_+(t,x) + k(x)n_+(t,x) &\le \int_{\qpn}K(y,x-y)n_+(t,y)d^ny.
\end{align*}
and integrating (using  Lebesgue Dominated Convergence Theorem)
\begin{align*}
\frac{d}{dt} \int_{\qpn} n_+(t,x)d^n x &\le \int_{\qpn} \int_{\qpn}K(y,x-y)n_+(t,y)d^ny  d^n x - \int_{\qpn}k(x)n_+(t,x)d^nx\\
&\le \int_{\qpn}n_+(t,y) \int_{\qpn} K(y,x-y)d^nx  d^n y - \int_{\qpn}k(x)n_+(t,x)d^nx\\
&\le \int_{\qpn}k(y)n_+(t,y)d^ny-\int_{\qpn}k(x)n_+(t,x)d^nx=0.
\end{align*}
So, we have $P[n_+]\ge 0$.  Also, since $\frac{d}{dt} \int_{\qpn} n_+(t,x)d^n x\le 0$ (the function is decreasing) we have for $t\ge 0$
\begin{equation}\label{eq-1}
\int_{\qpn}n_+(t,x)d^nx \le \int_{\qpn}n_+(0,x)d^n x =\int_{\qpn}n_+^0(x)d^n x.
\end{equation}
If $n^0$ is non-positive, then $n_+^0 (x)=0$ and $\int_{\qpn} n_+^0(x)d^nx=0$, and by equation \eqref{eq-1} $\int_{\qpn}n_+(t,x)d^nx\le 0$ therefore $n_+(t,x)=0$ a.e., which implies $n(t,x)\le 0$ a.e..  It means that when the initial data $n^0$ is non-positive, the solution $n$ is also non-positive,  applying this to $-n^0$ we conclude that
\begin{equation*}
n^0\ge 0 \Rightarrow n\ge 0.
\end{equation*}

\t Since $-n+2n_+=|n|$ we arrive to \eqref{eq--9.5}
\begin{align*}
\int_{\qpn}|n(t,x)|d^nx &=\int_{\qpn}-n(t,x)d^nx + 2\int_{\qpn}n_+(t,x)d^n x \notag\\
&\le \int_{\qpn}-n_0(t,x)d^nx + 2\int_{\qpn}n_+^0(t,x)d^n x\notag \\
&\le \int_{\qpn}|n^0(t,x)|d^nx. 
\end{align*}
\end{proof}

\t Properties \eqref{eq--9.3}, \eqref{eq--9.4} and \eqref{eq--9.5} are similar to the ones given for the Fokker-Planck equation.

\section{The relative entropy}
In order to prove the main theorem of this section we have to consider the dual problem to \eqref{scattering}, which can be written as
\begin{equation}\label{dual}
-\frac{\partial}{\partial t} \phi(t,x) + k(x)\phi(t,x) = \int_{\qpn}K(x, y-x) \phi(t,y) d^n y, \qquad t\ge 0, x\in \qpn.
\end{equation}
We assume that there are solutions $N(x)>0$ and $\phi(x)>0$ to the primal equation \eqref{scattering} and the dual equation \eqref{dual} respectively,  namely
\begin{equation}\label{eq-6.29}
k(x)N(x)=\int_{\qpn} K(y, x-y)N(y)d^ny
\end{equation}
\begin{equation}\label{eq-6.30}
k(x)\phi(x)=\int_{\qpn} K(x, y-x)\phi(y)d^ny.
\end{equation}
and we suppose here that these two solutions are independent of the time.
These two steady state solutions allow us to derive the general relative entropy inequality.

\begin{example}[Projection operator]
	Lets consider $N(x)>0$  and choose a weight $\bar g$ satisfying $\int_{\qpn}\bar{g}(y)N(y)d^ny =1$ and take $K$ as
	$$K(y,x-y)=\bar{g}(y)k(x)N(x).$$
	Then 
	\begin{align*}
	\int_{\qpn}K(y,x-y)N(y)d^ny &=\int_{\qpn}\bar{g}(y)k(x)N(x)N(y) d^ny\\
	&=k(x)N(x)\int_{\qpn}\bar{g}(y)N(y)d^ny= k(x)N(x).
	\end{align*}
\end{example}

\begin{example}
	Lets consider $N(x)>0$ with $\int_{\qpn}N(x)d^n x=1$ and a symmetric kernel $\tilde{K}(x,y)=\tilde{K}(y,x)>0$, and
	$$K(y,x-y)=\frac{\tilde{K}(x,y)}{N(y)}, \qquad k(y):=\int_{\qpn} \frac{\tilde{K}(x,y)}{N(y)} d^nx.$$
	Then 
	\begin{align*}
	\int_{\qpn}K(y,x-y)N(y)d^ny &= \int_{\qpn}\frac{\tilde K (x,y)}{N(y)} N(y) d^ny\\
	&=\int_{\qpn}\tilde{K}(x,y) d^ny = \int_{\qpn}\tilde K(y,x) d^n y\\
	&=k(x)N(x).
		\end{align*}
\end{example}

\begin{example}
	Suppose that there exists a function $N(x)>0$ such that the scattering cross-section satisfies the symmetry condition (usually called detailed balance or micro-reversibility)
	$$K(y,x)N(x)=K(x,y) N(y).$$
	Since $k(y)=\int_{\qpn}K(x,y)d^n x$ we have that
	\begin{align*}
	\int_{\qpn}K(x,y) N(y)d^n y &= \int_{\qpn}K(y,x) N(x) d^ny \\
	&= N(x) \int_{\qpn}K(y,x) d^ny = N(x) k(x).
	\end{align*}
	In this example the solution to the dual equation \eqref{eq-6.30} is $\phi(x)=1.$
\end{example}

\begin{lemma}[General Relative Entropy for Scattering equation]\label{Lema-GRE}
	Let $N(x)$ and $n(t,x)$ be  solutions to the primal equation \eqref{scattering}, and let $\phi(x)$ be a solution to the dual equation \eqref{dual}.  For any function $H:\r \to \r$ we have
	\begin{align*}
&\frac{\partial}{\partial t}\left[\phi(x)N(x) H\left(\frac{n(t,x)}{N(x)}\right) \right] \\
&\hspace{1cm} + \int_{\qpn} K(y, x-y) \left[\phi(y)N(x)H\left(\frac{n(t,x)}{N(x)}\right) - \phi(x)N(y) H\left(\frac{n(t,y)}{N(y)}\right) \right] d^ny\\
&=\int_{\qpn}K(y,x-y)\phi(x) N(y) \left[ H'\left(\frac{n(t,x)}{N(x)}\right) \left[\frac{n(t,y)}{N(y)} -\frac{n(t,x)}{N(x)}\right]\right.\\
&\hspace{2cm}+\left. H\left(\frac{n(t,x)}{N(x)}\right)- H\left(\frac{n(t,y)}{N(y)}\right)  \right]d^ny.	
\end{align*}
and 
	\begin{align*}
\frac{d}{dt}\int_{\qpn}\phi(x)&N(x) H\left(\frac{n(t,x)}{N(x)}\right) d^n x \\
&=\int_{\qpn}\int_{\qpn}K(y,x-y)\phi(x) N(y) \left[ H'\left(\frac{n(t,x)}{N(x)}\right) \left[\frac{n(t,y)}{N(y)} -\frac{n(t,x)}{N(x)}\right]\right.\\
&\hspace{2cm}\left.+ H\left(\frac{n(t,x)}{N(x)}\right)- H\left(\frac{n(t,y)}{N(y)}\right)  \right]d^ny d^nx.
\end{align*}
\end{lemma}
\begin{proof}
Following the proof given in \cite{Transport}, for the scattering equation \eqref{scattering} we calculate the entropy production.
	\begin{align*}
	\frac{d}{dt}&\left[\phi(x)N(x) H\left(\frac{n(t,x)}{N(x)}\right) \right] \\
	&= \phi(x)N(x) H'\left(\frac{n(t,x)}{N(x)}\right)\frac{1}{N(x)} \frac{\d}{\d t} n(t,x)\\
	&=\phi(x)H'\left(\frac{n(t,x)}{N(x)}\right) \left[ -k(x) n(t,x) + \int_{\qpn}K(y, x-y) n(t,y)d^n y \right]\\
	&=-H'\left(\frac{n(t,x)}{N(x)}\right)\frac{n(t,x)}{N(x)} \phi(x)\int_{\qpn}K(y, x-y)N(y)d^ny \\
	&+ H'\left(\frac{n(t,x)}{N(x)}\right)\phi(x)\int_{\qpn}K(y, x-y) n(t,y)d^n y\\
   &=\int_{\qpn}K(y,x-y)\phi(x) N(y) \left[ H'\left(\frac{n(t,x)}{N(x)}\right) \left[\frac{n(t,y)}{N(y)} -\frac{n(t,x)}{N(x)}\right]\right]d^ny,	
	\end{align*}
therefore
	\begin{align*}
\frac{d}{dt}&\left[\phi(x)N(x) H\left(\frac{n(t,x)}{N(x)}\right) \right] \\
&\hspace{1cm}+ \int_{\qpn} K(y, x-y) \left[\phi(y)N(x)H\left(\frac{n(t,x)}{N(x)}\right) - \phi(x)N(y) H\left(\frac{n(t,y)}{N(y)}\right) \right] d^ny\\
   &=\int_{\qpn}K(y,x-y)\phi(x) N(y) \left[ H'\left(\frac{n(t,x)}{N(x)}\right) \left[\frac{n(t,y)}{N(y)} -\frac{n(t,x)}{N(x)}\right]\right.\\
   &\hspace{1cm}\left.+ H\left(\frac{n(t,x)}{N(x)}\right)- H\left(\frac{n(t,y)}{N(y)}\right)  \right]d^ny.	
	\end{align*}

After integration in $x$ we have
	\begin{align*}
\frac{d}{dt}&\int_{\qpn}\phi(x)N(x) H\left(\frac{n(t,x)}{N(x)}\right) d^n x \\
&=\int_{\qpn}\int_{\qpn}K(y,x-y)\phi(x) N(y) \left[ H'\left(\frac{n(t,x)}{N(x)}\right) \left[\frac{n(t,y)}{N(y)} -\frac{n(t,x)}{N(x)}\right]\right.\\
&\hspace{2cm}\left.+ H\left(\frac{n(t,x)}{N(x)}\right)- H\left(\frac{n(t,y)}{N(y)}\right)  \right]d^ny d^nx.
\end{align*}


Since the integral 
\begin{align*}
&\int_{\qpn} \int_{\qpn} K(y, x-y) \left[\phi(y)N(x)H\left(\frac{n(t,x)}{N(x)}\right) - \phi(x)N(y) H\left(\frac{n(t,y)}{N(y)}\right) \right] d^nyd^nx =0.
\end{align*}
\end{proof}

\begin{theorem}\label{entropy-ineq}
	In the conditions of the previous lemma, and for $\phi(x)=1$ we have for any  convex function $H:\r \to \r$ there holds
	\begin{align}\label{entropy-production}
&\frac{d}{dt}\left[N(x) H\left(\frac{n(t,x)}{N(x)}\right) \right] \\
&\hspace{2cm} + \int_{\qpn} K(y, x-y) \left[N(x)H\left(\frac{n(t,x)}{N(x)}\right) - N(y) H\left(\frac{n(t,y)}{N(y)}\right) \right] d^ny \notag \\
&=\int_{\qpn}K(y,x-y) N(y) \left[ H'\left(\frac{n(t,x)}{N(x)}\right) \left[\frac{n(t,y)}{N(y)} -\frac{n(t,x)}{N(x)}\right]\right.\\
&\hspace{2cm}\left.+ H\left(\frac{n(t,x)}{N(x)}\right)- H\left(\frac{n(t,y)}{N(y)}\right)  \right]d^ny.	\notag
\end{align}	
and 
\begin{align}\label{entropy}
\frac{d}{dt}\int_{\qpn}&N(x) H\left(\frac{n(t,x)}{N(x)}\right) d^n x \\
&=\int_{\qpn}\int_{\qpn}K(y,x-y) N(y) \left[ H'\left(\frac{n(t,x)}{N(x)}\right) \left[\frac{n(t,y)}{N(y)} -\frac{n(t,x)}{N(x)}\right]\right.\\
&\hspace{1cm}\left.+ H\left(\frac{n(t,x)}{N(x)}\right)- H\left(\frac{n(t,y)}{N(y)}\right)  \right]d^ny d^nx. \notag
\end{align}	
Therefore	
	\begin{equation*}
	\frac{d}{dt}\int_{\qpn}N(x) H\left(\frac{n(t,x)}{N(x)}\right) d^n x \le 0.
	\end{equation*}
\end{theorem}
\begin{proof}
	We easily can check that  $\phi(x)=1$ is a solution of the dual equation \eqref{dual} (that means the primal equation is conservative), then equation \eqref{eq-6.30} correspond to equation \eqref{relacionK-k} and we obtain the result.

	Finally, since the function $H:\r \to \r$ is convex, we have that 
	\begin{equation}\label{convex}
	H(v)-H(u)\ge H'(u)(v-u).
	\end{equation}
	Therefore, it leads to
	\begin{equation*}
	0\ge H'\left(\frac{n(t,x)}{N(x)}\right) \left[\frac{n(t,y)}{N(y)} -\frac{n(t,x)}{N(x)}\right]+ H\left(\frac{n(t,x)}{N(x)}\right)- H\left(\frac{n(t,y)}{N(y)}\right) 
	\end{equation*}
	which shows that 
	\begin{equation*}
	\frac{d}{dt}\int_{\qpn}N(x) H\left(\frac{n(t,x)}{N(x)}\right) d^n x \le 0.
	\end{equation*}
\end{proof}

\t In other words we have obtained that 
\begin{equation*}
t \mapsto \mathcal{H}(n|N)(t):= \int_{\qpn} N(x) H\left(\frac{n(t,x)}{N(x)}\right) d^nx \quad \text{ is decreasing.}
\end{equation*} 
Up to our knowledge, this entropy principle is only known in conservative cases.

\begin{corollary}
	Assume that $n^0(x)\le C^0 N(x)$, then for all $t\ge 0$,
	\begin{equation*}
	n(t,x)\le C^0 N(x).
	\end{equation*}
\end{corollary}

\begin{proof}
	Assume $\frac{n^0(x)}{N(x)}\le C^0$ for some constant $C^0$, and suppose that for some $t>0$ we have $\frac{n(t,x)}{N(x)}> C^0$.  Since the function $\mathcal{H}(n|N)(t)$ is decreasing, we have that
	\begin{align*}
	\int_{\qpn}N(x) H\left(\frac{n(t,x)}{N(x)}\right)dx \le 	\int_{\qpn}N(x) H\left(\frac{n^0(x)}{N(x)}\right)dx 
	\end{align*}
	or
	\begin{align*}
\int_{\qpn}N(x) \left[ H\left( \frac{n(t,x)}{N(x)}\right) - H\left(\frac{n^0(x)}{N(x)} \right)\right]d^nx \le 0
\end{align*}
for any convex function.  In particular, if we take $H(u)=u$ we obtain
	\begin{align*}
	 \int_{\qpn}\underbrace{N(x)}_{>0} \underbrace{\left( \frac{n(t,x)}{N(x)} - \frac{n^0(x)}{N(x)} \right)}_{>0}d^nx\le 0,
\end{align*}
which is a contradiction.  We conclude that $\frac{n(t,x)}{N(x)}\le C^0$ for all $t>0.$
\end{proof}


\section{Exponential time decay}
In this section we wonder if the solutions to the system \eqref{scattering}-\eqref{relacionK-k} converge as $t\to \infty$, and if so, how fast?  Before to write the main result of this section, we need to define the {\it steady state} of the system.  More precisely, we want to show that if we put $\phi=1$ in equation \eqref{rho}, the solutions $n(t,x)\to \rho N(x)$ as $t\to \infty$ in some sense.

In order to define  such $\rho$ we put $H(u)=u$ in the second part of Lemma \ref{Lema-GRE}, then we obtain $\frac{d}{dt} \int_{\qpn} \phi(x) n(t, x)d^n x=0$ which means that $\int_{\qpn} \phi(x) n(t, x)d^n x$ is constant for $t$, i.e. 	
	\begin{equation}\label{rho}
	\rho:=\int_{\qpn}\phi(x)n^0(x)d^nx=\int_{\qpn}\phi(x)n(t,x)d^nx.
	\end{equation}
	 
\t On the other hand, it is easy to see that $h(t,x):=n(t,x)-\rho N(x)$  is a solution of  the system \eqref{scattering}-\eqref{relacionK-k},  provided  $n(t,x)$ is a solution, in fact:
\begin{align*}
\frac{\d}{\d t}h(t,x) &+ k(x)h(t, x) - \int_{\qpn} K(y, x-y)h(t,y) d^n y \\
&= \frac{\d}{\d t}n(t,x) +\underbrace{\frac{\d}{\d t}\rho N(x) }_{=0}+ k(x)n(t,x) -k(x)\rho N(x)-\\
&\hspace{1cm}\int_{\qpn} K(y, x-y)[n(t,y)-\rho N(x)] d^n y \\
&= \frac{\d}{\d t}n(t,x) + k(x)n(t,x) -
\int_{\qpn} K(y, x-y)n(t,y)d^ny =0.
\end{align*}

\t When we think  about long time convergence, a control of entropy by entropy dissipation is useful for exponential convergence as $t\to \infty$. The following result can be seen as a {\it Poincaré inequality}.
\begin{lemma}[Analog to Lemma 6.2 in \cite{Transport}]\label{Poincare-inequality}
	Given $\phi(x)>0, \ N(x)>0, \ K(y,x-y)>0$, there exists a constant $\alpha>0$ such that for all test function $m(x)$ satisfying $$\int_{\qpn} \phi(x) m(x) d^n x=0,$$ we have 
{\small	$$ \int_{\qpn}\int_{\qpn}K(y,x-y) \phi(x) N(y) \left( \frac{m(x)}{N(x)} - \frac{m(y)}{N(y)}\right)^2 d^ny d^n x \ge \alpha \int_{\qpn} \phi(x) N(x) \left(\frac{m(x)}{N(x) }\right)^2 d^n x.$$}
\end{lemma}

\begin{proof}
	When the integral in the right side vanishes the result holds, so we can suppose that integral is non-zero and we  normalize it, i.e., we suppose that $\int_{\qpn} \phi(x) N(x) \left(\frac{m(x)}{N(x) }\right)^2 d^n x=1.$  Then we argue by contradiction.  If such an $\alpha$ does not exist, we can find a sequence of test functions $m_{k}(x)$ such that
	$$\int_{\qpn} \phi(x) m_{k}(x) d^n x=0, \quad \int_{\qpn} \phi(x) N(x) \left(\frac{m_{k}(x)}{N(x) }\right)^2 d^n x=1,$$ and $$ \int_{\qpn}\int_{\qpn}K(y,x-y) \phi(x) N(y) \left( \frac{m_{k}(x)}{N(x)} - \frac{m_{k}(y)}{N(y)}\right)^2 d^ny d^n x\le \frac{1}{k}.$$
	
\t	Consider the collection of test functions $\mathcal F=\{m_k(x) \}$.  Since each one of these functions has compact support, we can assume that all the functions have support in a ball $T$.  
	\begin{itemize}
		\item	This collection is uniformly equicontinuous on $T$, it means that for every $\epsilon >0$, there exists a $\delta>0$ such that 
	$$\sup_{||s-q||_p<\delta} |m_k(s)- m_k(q)|\le \epsilon \quad \text{for all }m_k\in \mathcal{F}.$$
	In fact, for $\epsilon>0$ and $m_k$ we can choose $\delta_k$  as the  parameter of constancy of $m_k$, and we have that $|m_k(s)-m_k(q)|=0$ for $||s-q||<\delta_k$.  Then we put $\delta:=\inf \delta_k>0$ and the inequality holds.
	
	\item The collection $\mathcal F$ is pointwise bounded, i.e. $\sup_{m_k\in \mathcal F} |m_k(s)|<\infty.$
	\end{itemize}
Then, by using the Arzelá-Ascoli theorem, we conclude that $\mathcal F\subset C(T, \r)$ has a convergent subsequence in the supremum norm. 	After the extraction of the subsequence, we may pass to the limit $m_{\epsilon}\to \bar{m}$ and this function satisfies
	$$\int_{\qpn} \phi(x) \bar{m}(x) d^n x=0, \quad \int_{\qpn} \phi(x) N(x) \left(\frac{\bar{m}(x)}{N(x) }\right)^2 d^n x=1,$$ and $$ \int_{\qpn}\int_{\qpn}K(y,x-y) \phi(x) N(y) \left( \frac{\bar{m}(x)}{N(x)} - \frac{\bar{m}(y)}{N(y)}\right)^2 d^ny d^n x=0.$$	
	From the last line, we conclude that $\frac{\bar{m}(x)}{N(x)} = \frac{\bar{m}(y)}{N(y)}:=\nu$.  Since 
	\begin{align*}
	\int_{\qpn} \phi(x) \bar{m}(x) d^n x&=0\\
	\int_{\qpn} \phi(x) N(x)\frac{\bar{m}(x)}{N(x)} d^n x&=0\\
	\nu \int_{\qpn} \phi(x) N(x)d^n x&=0,
	\end{align*}
	then $\nu=0$, in other words $\bar{m}=0$ which contradicts the normalization and thus such an $\alpha$ should exist.
\end{proof}

\begin{remark}To this point we do not have big differences with classical case, we just adapt the techniques to $p$-adic case, but we have some observations.
	\begin{enumerate}
		\item 	Lemma \ref{Poincare-inequality} is analog to Lemma 6.2 in \cite{Transport}, but we need to assume here that functions $m(x)$ are test functions, i.e. $m(x)\in \dd$, otherwise we cannot extract the sub-sequence.
		\item In order to apply the previous lemma to function $h(t,x)$ it is convenient to normalize some functions, more precisely we need to assume that $$\int_{\qpn} N(x)d^nx=1 \quad \text{and} \quad \int_{\qpn}N(x)\phi(x)d^nx=1.$$
		With the second condition we obtain that $\int_{\qpn} \phi(x) h(t,x)d^n x=0$, in fact
		\begin{align*}
		\int_{\qpn} \phi(x) h(t,x)d^n x&= \int_{\qpn} \phi(x) [n(t,x)- \rho N(x)]d^n x\\
		&=\int_{\qpn} \phi(x) n(t,x)d^n x-\rho \int_{\qpn}\phi(x) N(x)d^nx=0.
		\end{align*}
		\item In the case $\phi(x)=1$, we can take the function $h(t,x)\in \liz(\qpn)$, the $p$-adic Lizorkin space of test functions of the second kind.  Thus we can apply the lemma without assuming the normalization condition $\int_{\qpn}N(x)d^nx =1$.
\end{enumerate}
		
\end{remark}
\begin{proposition}
	For the solutions of the system \eqref{scattering}-\eqref{relacionK-k} we have,
	\begin{itemize}
		\item[(i)] 	$\rho:=\int_{\qpn}\phi(x)n^0(x)d^nx=\int_{\qpn}\phi(x)n(t,x)d^nx.$
		\item[(ii)] $\int_{\qpn} \phi(x) |n(t,x)|d^nx \le \int_{\qpn}\phi(x) |n^0(x)|d^nx$ for all $t>0.$
		\item[(iii)] If $C^1 N(x)\le n^0(x) \le C^0 N(x)$, then $C^1 N(x)\le  n(t,x) \le C^0 N(x)$ for all $t>0.$
		\item[(iv)]	If $n(t,x)$ and $N(x)$ are test functions, there exists a constant $\alpha>0$ such that 
		\begin{equation}\label{time-decay}
		\int_{\qpn}\phi(x) N(x) \left(\frac{n(t,x)}{N(x)} -\rho\right)^2 d^n x \le e^{-\alpha t} \int_{\qpn}\phi(x) N(x) \left(\frac{n^0(x)}{N(x)} -\rho\right)^2 d^n x
		\end{equation}
	\end{itemize}
\end{proposition}

\begin{proof}
	For (i) and (ii) we just choose the convex functions $H_1(u)=u$ and $H_2(u)=|u|$ respectively, in Lemma \ref{Lema-GRE}. For (iii) we choose $H_3(u)=(u-C^0)_+^2$ for the upper bound, and $H_4(u)=(C_1-u)_+^2$ for the lower bound. In order to prove the exponential time decay \eqref{time-decay} we choose the convex function  $H(u)=u^2$ in the second part of Lemma \eqref{Lema-GRE}, therefore 
	\begin{align*}
	\frac{d}{dt} \int_{\qpn}& \phi(x) N(x) \left(\frac{h(t,x)}{N(x)}\right)^2 d^n x \\
	&= \int_{\qpn}\int_{\qpn} K(y, x-y) \phi(x)N(y) \left[2\left(\frac{h(t,x)}{N(x)}\right) \left[\frac{h(t,y)}{N(y)} - \frac{h(t,x)}{N(x)}\right]\right.\\
	&\hspace{2cm}\left. + \left(\frac{h(t,x)}{N(x)}\right)^2 - \left(\frac{h(t,y)}{N(y)}\right)^2\right]d^ny d^n x\\
	&=- \int_{\qpn}\int_{\qpn} K(y, x-y) \phi(x) N(y) \left(  \frac{h(t,x)}{N(x)} - \frac{h(t,y)}{N(y)} \right)^2 d^ny d^n x\le 0,
	\end{align*}
	then, by using Lemma \ref{Poincare-inequality} we conclude that there exists $\alpha>0$ such that
	\begin{align*}
	\frac{d}{dt} \int_{\qpn} \phi(x) N(x) \left(\frac{h(t,x)}{N(x)}\right)^2 d^n x 
	&\le -\alpha  \int_{\qpn} \phi(x) N(x) \left(\frac{h(t,x)}{N(x) }\right)^2 d^n x
	\end{align*}
	
\t	By using the Gronwall's Lemma we conclude that
	$$\int_{\qpn} \phi(x) N(x) \left(\frac{h(t,x)}{N(x)}\right)^2 d^n x \le e^{-\alpha t} \int_{\qpn} \phi(x) N(x) \left(\frac{h(0,x)}{N(x)}\right)^2 d^n x  $$
	or
	$$\int_{\qpn} \phi(x) N(x) \left(\frac{n(t,x)}{N(x)} - \rho \right)^2 d^n x \le e^{-\alpha t} \int_{\qpn} \phi(x) N(x) \left(\frac{n^0(x)}{N(x)} - \rho \right)^2 d^n x.$$	
\end{proof}
\t The last part of the previous proposition says that the solutions to the system converge as $t\to \infty$ with an exponential rate, i.e. 
\begin{equation}
n(t,x) \to \rho N(x) \quad \text{ as }t\to \infty.
\end{equation}

\section{Time dependent coefficients}
The above manipulations are also valid for time dependent coefficients.  More precisely, when we consider the problem
\begin{equation}\label{scattering-t}
\frac{\d}{\d t}n(t,x) + k(t,x)n(t, x) = \int_{\qpn} K(t,y, x-y)n(t,y) d^n y, \qquad t\ge 0, x\in \qpn,
\end{equation}
with the initial data $n^0\in L^1(\qpn)$, and  $K(t,y,z)\ge 0,$  $k(t,y):=\int_{\qpn} K(t,y,z)d^nz \in L^{\infty}(\qpn)$ and the dual equation
\begin{equation}\label{dual-t}
-\frac{\partial}{\partial t} \phi(t,x) + k(t,x)\phi(t,x) = \int_{\qpn}K(t,x, y-x) \phi(t,y) d^n y, \qquad t\ge 0, x\in \qpn.
\end{equation}
we have the following entropy inequality.

\begin{lemma}
	Let $N(t,x)$ and $n(t,x)$ be  solutions to the primal equation \eqref{scattering-t}, and let $\phi(t,x)$ be a solution to the dual equation \eqref{dual-t}.  For any convex function $H:\r \to \r$ we have

	\begin{align*}
	\frac{d}{dt}\int_{\qpn}&\phi(t,x)N(t,x) H\left(\frac{n(t,x)}{N(t,x)}\right) d^n x \\
	&=\int_{\qpn}\int_{\qpn}K(y,x-y)\phi(t,x) N(t,y) \left[ H'\left(\frac{n(t,x)}{N(t,x)}\right) \left[\frac{n(t,y)}{N(t,y)} -\frac{n(t,x)}{N(t,x)}\right]\right.\\
	&\hspace{4cm}\left.+ H\left(\frac{n(t,x)}{N(t,x)}\right)- H\left(\frac{n(t,y)}{N(t,y)}\right)  \right]d^ny d^nx\\
	&\le 0.
	\end{align*}
\end{lemma}

\section{Hyperbolic Rescaling}
We assume that scattering occurs with small changes and a fast rate (in other words, we change the time scale according to the size of jumps), then we have the following problem
\begin{equation}\label{hyperbolic}
\begin{cases}
\frac{\partial}{\partial t} \nep(t,x) + \frac{1}{|\varepsilon|_p}\left[ k(x)\nep(t,x)- \int_{\qpn}\frac{1}{|\varepsilon|_p^n} K(y, \frac{x-y}{|\varepsilon|_p})  \nep (t,y) d^n y\right]=0,\\
\nep (0, x)=n^0(x), \qquad \int_{\qpn}|n^0(x)|d^nx=: M^0.
\end{cases}
\end{equation}

\t Consider a collection of test functions $\mathcal F=\{\nep(t,x) \}$ which are solutions of \eqref{hyperbolic}, then we can extract a sub-sequence that converges, in some sense, to a solution $n$ of \eqref{scattering}.

To do that, suppose that we have a function $\bar N\in L^1(\qpn)\cap L^{\infty}(\qpn)$ such that for the initial data and the steady states $N_{\varepsilon}$ undergo uniform control
$$|n^0|\le C_0 N_{\varepsilon}\le \bar N.$$
This implies the same control for all $t\ge 0$
$$|n(t,x)|\le C_0 N_{\varepsilon}\le \bar N.$$

\t By using the same argue as in Lemma \ref{Poincare-inequality}, we extract a convergent  sub-sequence in the supremum norm i.e. $\sup_{x\in \qpn}|\nep(t,x)- n(t,x)|\to 0$ for all $t\ge 0$.   Then for $\phi(t,x)\in L^1(\r^+\times \qpn)$
\begin{align*}
\int_{0}^{\infty} \int_{\qpn} \phi(t,x) &(\nep(t,x)- n(t,x)) d^nx dt 
\le  \int_{0}^{\infty} \int_{\qpn} \phi(t,x) |\nep(t,x)- n(t,x)| d^nx dt \\
&\le \sup_{x\in \qpn}|\nep(t,x)- n(t,x)| \int_{0}^{\infty} \int_{\qpn} |\phi(t,x)| d^nx dt,
\end{align*}

\t in other words
\begin{align*}
\int_{0}^{\infty} \int_{\qpn} \phi(t,x) \nep(t,x) d^nx dt \xrightarrow[\varepsilon\to 0]{} \int_{0}^{\infty} \int_{\qpn} \phi(t,x) n(t,x) d^nx dt,
\end{align*}
which means that 
\begin{equation*}
\nep \rightharpoonup n \text{  in  } L^{\infty}(\r^+\times \qpn)-w^*, \quad |n|\le \bar N.
\end{equation*}

\t We will show now that, under some regularity assumptions for $K$, we can derive stronger  bounds than $L^1$  for solutions of \eqref{hyperbolic}.

\begin{proposition}
	Consider $\nep$ a solution  of \eqref{hyperbolic}, and assume \eqref{relacionK-k} and
	\begin{equation}\label{regularity-K}
	\int_{\qpn}\left[ K(x-|\varepsilon|_p z, z) - K(x,z) \right] dz \le |\varepsilon|_p L_1,
	\end{equation}
	therefore 	for all $t\ge 0$
	\begin{equation*}
	 ||\nep (t,x) ||_{L^2(\qpn)} \le  e^{L_1 t}||n^0(x)||_{L^2(\qpn)}.
	\end{equation*}
\end{proposition}

\begin{proof}
We multiply \eqref{hyperbolic} by $\nep(t,x)$	
{\small\begin{align*}
\nep(t,x)\frac{\partial}{\partial t} \nep(t,x) + \frac{1}{|\varepsilon|_p}\left[ k(x)\nep(t,x)^2- \int_{\qpn}\frac{1}{|\varepsilon|_p^n} K(y, \frac{x-y}{|\varepsilon|_p})  \nep (t,y)\nep(t,x) d^n y\right]=0,
\end{align*}}
and integrate with respect to $x$
\begin{align*}
\int_{\qpn} \frac12 \frac{d}{d t} \nep(t,x)^2 d^nx  &+ \int_{\qpn} \frac{1}{|\varepsilon|_p} k(x)\nep(t,x)^2 d^n x\\
&=\int_{\qpn} \int_{\qpn}\frac{1}{|\varepsilon|_p^{n+1}} K(y, \frac{x-y}{|\varepsilon|_p})  \nep (t,y)\nep(t,x) d^n y d^n x.
\end{align*}
After the change of variables $z=\frac{x-y}{|\varepsilon|_p}$ we have
\begin{align*}
 \frac12 \frac{d}{d t} \int_{\qpn}& \nep(t,x)^2 d^nx  + \int_{\qpn}\int_{\qpn} \frac{1}{|\varepsilon|_p} K(x,z)\nep(t,x)^2 d^nz d^n x\\
 &= \int_{\qpn} \int_{\qpn}\frac{1}{|\varepsilon|_p^{n+1}} K(x-|\varepsilon|_p z, z)  \nep (t,x-|\varepsilon|_p z)\nep(t,x)|\varepsilon|_p^n d^n z d^n x\\
&\le \frac12 \int_{\qpn} \int_{\qpn}\frac{1}{|\varepsilon|_p} K(x-|\varepsilon|_p z, z)  \left( \nep (t,x-|\varepsilon|_p z)^2 + \nep(t,x)^2 \right) d^n z d^n x. 
\end{align*}
Since $\int_{\qpn} K(x-|\varepsilon|_p z, z)   \nep (t,x-|\varepsilon|_p z)^2  d^nx =\int_{\qpn}  K(x,z)\nep(t,x)^2 d^n x$ it holds

\begin{align*}
\frac12 \frac{d}{d t} \int_{\qpn} \nep(t,x)^2 d^nx &\le \frac12 \int_{\qpn}\int_{\qpn} \frac{K(x-|\varepsilon|_p z,z) -K(x,z)}{|\varepsilon|_p} \nep (t,x)^2 d^nz d^n x\\
\frac{d}{d t} \int_{\qpn} \nep(t,x)^2 d^nx &\le  \int_{\qpn} \nep (t,x)^2  \int_{\qpn} \frac{K(x-|\varepsilon|_p z,z) -K(x,z)}{|\varepsilon|_p} d^nz d^n x\\
&\le L_1 \int_{\qpn} \nep (t,x)^2 d^nx. 
\end{align*}
By Gronwall's lemma we can conclude that 
\begin{equation*}
\int_{\qpn} \nep (t,x)^2 d^nx \le e^{L_1 t}\int_{\qpn} \nep^0 (x)^2 d^nx,
\end{equation*}
in other words
\begin{equation*}
||\nep ||_{L^2(\qpn)} \le e^{L_1 t} ||\nep^0||_{L^2(\qpn)} = e^{L_1 t} ||n^0||_{L^2(\qpn)}.
\end{equation*}
\end{proof}

\section{Stability}
\begin{theorem}[Analog to Smoller, Theorem 16.1 (c), p. 266]
The solution $n(t,x)$ is stable in the following sense:  If $n_0, v_0\in L^1(\qpn)$ and $v$ is the corresponding constructed solution of \eqref{scattering} with initial data $v_0$, then 
\begin{equation}\label{stability}
\int_{\qpn} |n(t,x)-v(t,x)|d^nx \le \int_{\qpn} |n^0(x)-v^0(x)|d^nx.
\end{equation} 
\end{theorem}

\begin{proof}
	It is clear that if $n(t,x)$ and $v(t,x)$ are solutions of \eqref{scattering}, then $n(t,x)-v(t,x)$ is also a solution.  In addition, if we take $H(u)=|u|$ in Theorem \eqref{entropy-ineq} and we apply it to $n(t,x)-v(t,x)$ we have that
	$$\frac{d}{dt}\int_{\qpn} N(x)\left| \frac{n(t,x)-v(t,x)}{N(x)}\right| d^nx \le 0,$$
	in other words the function $G(t)=\int_{\qpn} \left| n(t,x)-v(t,x)\right| d^nx$ is decreasing, thus for $t>0$ we have $G(t)\le G(0)$, or 
	\begin{equation}
	\int_{\qpn} |n(t,x)-v(t,x)|d^nx \le \int_{\qpn} |n^0(x)-v^0(x)|d^nx.
	\end{equation} 
\end{proof}

Observe that the entropy inequality lead us to uniqueness.

\end{document}